\documentclass{article}
\usepackage{amsmath,amsthm,amssymb}

\theoremstyle{definition}
\newtheorem{Thm}{Theorem}[section]

\newtheorem{Lem}[Thm]{Lemma}

\newtheorem{Prop}[Thm]{Proposition}
\newtheorem{Rem}[Thm]{Remark}
\newtheorem{Quest}[Thm]{Question}

\DeclareMathOperator{\Homeo}{Homeo}

\DeclareMathOperator{\rot}{rot}
\DeclareMathOperator{\PSL}{PSL}
\DeclareMathOperator{\SL}{SL}
\DeclareMathOperator{\Tr}{tr}
\DeclareMathOperator{\Rep}{R}
\DeclareMathOperator{\id}{id}
\DeclareMathOperator{\Int}{Int}
\newcommand{\C}{\ensuremath{\mathbb{C}}}
\newcommand{\R}{\ensuremath{\mathbb{R}}}

\newcommand{\Z}{\ensuremath{\mathbb{Z}}}

\title{Rotation number and lifts of a Fuchsian action of the modular group on the circle}
\date{}
\author{Yoshifumi MATSUDA\footnote
{Department of Physics and Mathematics, College of Science and Engineering, 
Aoyama Gakuin University 5-10-1 Fuchinobe, Chuo-ku, Sagamihara-shi, Kanagawa 252-5258 Japan, 
ymatsuda@gem.aoyama.ac.jp} \footnote{Partly supported by JSPS KAKENHI Grant Number 25800036.}
}
\begin{document}

\maketitle

\begin{abstract}
We characterize the semi-conjugacy class of a Fuchsian action of the modular group on the circle in terms of rotation numbers of two standard generators and that of their product. 
We also show that among lifts of a Fuchsian action of the modular group, only 5-fold lift admits a similar characterization. 
These results indicate similarity and difference between rotation number and linear character. \\



\end{abstract}

\section{Introduction}
Rotation number of an orientation-preserving homeomorphism of the circle has similar properties to absolute value of the trace of an element in $\PSL(2,\R)$. 
For example, they are invariant under conjugation and furthermore, J{\o}rgensen's criterion of discreteness for subgroups of $\PSL(2,\R)$ \cite[Theorem 2]{Jor77}, which can be described in terms of absolute value of the trace, has an analogue for the group of real analytic diffeomorphisms of the circle (see \cite[Theorem 1.2]{Mat09}). 
In this article, we give another similarity between rotation number and linear character from a viewpoint given by D. Calegari and A. Walker \cite{CW11}. 

\subsection{Rotation number}
We denote by $\Homeo_+(\mathbb{S}^1)$ the group of orientation-preserving homeomorphisms of the circle. 
We regard the circle $\mathbb{S}^1$ as the quotient $\R\slash\Z$ and denote by $p \colon \R \to \mathbb{S}^1$ the projection. 
Let $\widetilde{\Homeo}_+(\mathbb{S}^1)$ be the group of lifts of orientation-preserving homeomorphisms to $\R$, namely, homeomorphisms of $\R$ commuting with integral translations. 

For $\tilde{f}\in\widetilde{\Homeo}_+(\mathbb{S}^1)$, we define the translation number $\widetilde{\rot}(\tilde{f})\in\R$ of $\tilde{f}$ by 
\begin{eqnarray*}
\widetilde{\rot}(\tilde{f})=\displaystyle\lim_{n\to\infty}\dfrac{(\tilde{f})^n(\tilde{x})-\tilde{x}}{n}, 
\end{eqnarray*}
where $\tilde{x}\in\R$. 
Note that the limit exists and does not depend on the choice of a point $\tilde{x}\in\R$.
For $f\in\Homeo_+(\mathbb{S}^1)$, we define the rotation number $\rot(f)\in\R\slash\Z$ of $f$ by 
\begin{eqnarray*}
\rot(f)=\widetilde{\rot}(\tilde{f})\mod\Z,
\end{eqnarray*}
where $\tilde{f}\in\widetilde{\Homeo}_+(\mathbb{S}^1)$ is a lift of $f$ to $\R$.

Among several properties of rotation number, we recall that $\rot(f)=\dfrac{p}{q}$, where $\dfrac{p}{q}$ is a reduced fraction if and only if $f$ has a period point of period $q$. 
In particular, $\rot(f)=0$ if and only if $f$ has a fixed point (see for example \cite{Ghy01} in detail and other properties of rotation number).  

\subsection{Lifts of a group action on the circle}
For a group $\Gamma$, we denote by $\Rep(\Gamma)$ the space of homomorphisms from $\Gamma$ to $\Homeo_+(\mathbb{S}^1)$. 
We equip $\Rep(\Gamma)$ with the uniform convergence topology on generators if necessary. 

We define a lift of a group action on the circle. 

Let $k\ge 2$ be a positive integer and denote by $p_k\colon\mathbb{S}^1\to\mathbb{S}^1$ the $k$-fold covering map. 
For a group $\Gamma$, a homomorphism $\phi\in\Rep(\Gamma)$ is a $k$-fold lift of a homomorphism $\psi\in\Rep(\Gamma)$ if $p_k\circ\phi(\gamma)=\psi(\gamma)\circ p_k$ for every $\gamma\in\Gamma$. 

We remark that if $\phi\in\Rep(\Gamma)$ is a $k$-fold lift of a homomorphism $\psi\in\Rep(\Gamma)$, then we have $k\rot(\phi(\gamma))=\rot(\psi(\gamma))$ for every $\gamma\in\Gamma$. 

\subsection{Semi-conjugacy class}
Semi-conjugacy between two actions of a group on the circle has been defined in several ways (see \cite{Ghy87}, \cite{Ghy01}, \cite{Buc08}). 
In this paper, we follow the way presented in \cite{Cal07}. 

For $\phi_1, \phi_2\in\Rep(\Gamma)$, we say that $\phi_1$ is {\it{semi-conjugate}} to $\phi_2$ if there exists a continuous degree-one monotone map such that 
$h\circ\phi_1(\gamma)=\phi_2(\gamma)\circ h$ 
for every $\gamma\in\Gamma$. 
Here, a map $h \colon \mathbb{S}^1 \to \mathbb{S}^1$ is called a degree-one monotone map if it admits a lift $\tilde{h} \colon \R \to \R$ commuting with integral translations, and nondecreasing on $\R$.

Note that semi-conjugacy is not symmetric and is not an equivalence relation. 
We consider the equivalence relation generated by semi-conjugacy, which is called monotone equivalence in \cite{Cal07}. 
We call the monotone equivalence class of $\phi\in\Rep(\Gamma)$ the {\it{semi-conjugacy class}} of $\phi$. 
Note that if two minimal homomorphisms belong to the same semi-conjugacy class, then they are topologically conjugate. 
We define the semi-conjugacy class of an orientation-preserving homeomorphism of the circle in a similar way. 

A classical result due to H. Poincar\'e says that two homeomorphisms are in the same semi-conjugacy class if and only if their rotation numbers coincide, which  is similar to the fact that two matrices in $\SL(2,\R)\setminus\{\pm E\}$ are conjugate if and only if their traces coincide. 

As for group actions, however, $\phi_1, \phi_2\in\Rep(\Gamma)$ do not belong to the same semi-conjugacy class if we only suppose that $\rot(\phi_1(\gamma))=\rot(\phi_2(\gamma))$ for every $\gamma$. 
It can be seen by considering Fuchsian actions corresponding to hyperbolic structures on 2-orbifolds (see for example \cite{Cho12} about 2-orbifolds and hyperbolic structures on them). 

\subsection{Fuchsian actions}
Let $\mathcal{O}$ be a compact, connected, oriented 2-orbifold with negative orbifold Euler characteristic $\chi^{orb}(\mathcal{O})<0$. 
For each hyperbolic structure on the interior of $\mathcal{O}$ compatible with the orientation of $\mathcal{O}$, we have a homomorphism from the orbifold fundamental group $\pi_1^{orb}(\mathcal{O})$ to $\PSL(2,\R)$ by identifying the universal cover $\tilde{\mathcal{O}}$ with the hyperbolic plane $\mathbb{H}^2$. 
By considering the action on the ideal boundary $\partial \mathbb{H}^2\simeq \mathbb{S}^1$, we obtain a homomorphism $\phi_{\mathcal{O}}\in\Rep(\pi_1^{orb}(\mathcal{O}))$. 
We call such a homomorphism a {\it{Fuchsian action}} associated to $\mathcal{O}$. 
Note that the semi-conjugacy class of a Fuchsian action associated to a fixed 2-orbifold $\mathcal{O}$ is independent of the choice of a hyperbolic structure and that a Fuchsian action corresponding to a hyperbolic structure with finite area is minimal. 

In general, we cannot characterize the semi-conjugacy class of a Fuchsian action only by rotation numbers of all elements. 
In fact, for a Fuchsian action $\phi_S$ associated to a compact, connected, oriented surface $S$ with negative Euler characteristic, the homeomorphism $\phi_{S}(\gamma)$ has a fixed point for every $\gamma\in\Gamma$ but there is no global fixed point.
This means that $\rot(\phi_{S}(\gamma))=0$ for every $\gamma\in\Gamma$ but the Fuchsian action $\phi_{S}$ does not belong to the semi-conjugacy class of the trivial action. 

Now we show, however, that we can characterize the semi-conjugacy classes of a Fuchsian action of a specific 2-orbifold and its certain lift by only rotation numbers of finite elements. 

\subsection{Main result}
We focus on a special 2-orbifold. 
Let $\mathcal{O}_{2,3}$ be the 2-orbifold which is obtained from a 2-disk by making two cone-points of orders $2,3$. 
Note that the interior of $\mathcal{O}_{2,3}$ is homeomorphic to $\mathbb{H}^2\slash\PSL(2;\mathbb{Z})$ and $\pi_1^{orb}(\mathcal{O}_{2,3})$ is isomorphic to the modular group $\PSL(2,\Z)$. 
We fix a presentation 
\begin{eqnarray*}
\pi_1^{orb}(\mathcal{O}_{2,3})=\langle \alpha,\beta ~|~ \alpha^2=\beta^3=1\rangle\cong\Z_2\ast\Z_3, . 
\end{eqnarray*}
where $\alpha=\left[
\begin{array}{cc}
0 & -1 \\
1 & 0
\end{array}
\right]$ and $\beta=\left[
\begin{array}{cc}
1 & 1 \\
-1 & 0
\end{array}
\right]$. 
Let $\phi_{\mathcal{O}_{2,3}}$ be a Fuchsian action of $\mathcal{O}_{2,3}$ which is equal to the action by linear fractional transformations on $\R\cup\{\infty\}\simeq\mathbb{S}^1$. 
It follows that
\begin{eqnarray*}
\phi_{\mathcal{O}_{2,3}}(\alpha)(0)=\infty,& 
\phi_{\mathcal{O}_{2,3}}(\alpha)(\infty)=0 & \\
\phi_{\mathcal{O}_{2,3}}(\beta)(0)=\infty,\  &
\phi_{\mathcal{O}_{2,3}}(\beta)(\infty)=-1, &
\phi_{\mathcal{O}_{2,3}}(\beta)(-1)=0\ {\rm{and}} \\
\phi_{\mathcal{O}_{2,3}}(\alpha\beta)(0)=0. & & 
\end{eqnarray*}
Hence we have
\begin{eqnarray*}
(\rot(\phi_{\mathcal{O}_{2,3}}(\alpha)), \rot(\phi_{\mathcal{O}_{2,3}}(\beta)), \rot(\phi_{\mathcal{O}_{2,3}}(\alpha\beta)))=\left(\dfrac{1}{2},\dfrac{1}{3}, 0\right). 
\end{eqnarray*}

It follows from the presentation of $\pi_1^{orb}(\mathcal{O}_{2,3})$ that there exists a $k$-fold lift $\phi_{\mathcal{O}_{2,3}}^{(k)}$ of $\phi_{\mathcal{O}_{2,3}}$ if and only if $k\equiv\pm1 \mod 6$ and that such a lift is unique if it exists. 
We also have 
\begin{eqnarray*}
& & (\rot(\phi_{\mathcal{O}_{2,3}}^{(k)}(\alpha)), \rot(\phi_{\mathcal{O}_{2,3}}^{(k)}(\beta)), \rot(\phi_{\mathcal{O}_{2,3}}^{(k)}(\alpha\beta))) \\
&=& \left\{
\begin{array}{ll}
\left(\dfrac{1}{2},\dfrac{1}{3}, \dfrac{k-1}{k}\right) & (k\equiv1 \mod 6) \\
\left(\dfrac{1}{2},\dfrac{2}{3}, \dfrac{1}{k}\right) & (k\equiv-1 \mod 6)
\end{array}
\right.
\end{eqnarray*}

Now we are ready to state the main result. 

\begin{Thm}
\label{main}
Let $\phi\in\Rep(\pi_1^{orb}(\mathcal{O}_{2,3}))$. 
\begin{enumerate}
    \item If $(\rot(\phi(\alpha)), \rot(\phi(\beta)), \rot(\phi(\alpha\beta))=\left(\dfrac{1}{2},\dfrac{1}{3}, 0\right)$, then $\phi$ belongs to the semi-conjugacy class of a Fuchsian action $\phi_{\mathcal{O}_{2,3}}$. 
    \item If $(\rot(\phi(\alpha)), \rot(\phi(\beta)), \rot(\phi(\alpha\beta))=\left(\dfrac{1}{2},\dfrac{2}{3}, \dfrac{1}{5}\right)$, then $\phi$ belongs to the semi-conjugacy class of the 5-fold lift $\phi_{\mathcal{O}_{2,3}}^{(5)}$ of a Fuchsian action $\phi_{\mathcal{O}_{2,3}}$. 
\end{enumerate}
\end{Thm}

\begin{Rem}
\begin{enumerate}
\item Theorem \ref{main} cannot be generalized to the other lifts of $\phi_{\mathcal{O}_{2,3}}$. 
Indeed for each positive integer $k\ge 2$ we denote by $\mathcal{O}_{2,3,k}$ a compact, connected, oriented 2-orbifold which is obtained from a 2-sphere by making three cone-points of orders 2,3,$k$. 
Now suppose that $k\equiv\pm 1\mod 6$ and $k\neq 5$. 
Then we have $\chi^{orb}(\mathcal{O}_{2,3,k})<0$. 
Let $\phi_{\mathcal{O}_{2,3,k}}\in\Rep(\pi_1^{orb}(\mathcal{O}_{2,3,k}))$ be a Fuchsian action of $\mathcal{O}_{2,3,k}$. 
For a suitable presentation 
\begin{eqnarray*}
\pi_1^{orb}(\mathcal{O}_{2,3,k})&=&\langle \alpha,\beta, \gamma ~|~ \alpha^2=\beta^3=\gamma^k=\alpha\beta\gamma=1\rangle,  
\end{eqnarray*}
we have 
\begin{eqnarray*}
& & (\rot(\phi_{\mathcal{O}_{2,3,k}}(\alpha)), \rot(\phi_{\mathcal{O}_{2,3,k}}(\beta)), \rot(\phi_{\mathcal{O}_{2,3,k}}(\gamma))) \\
&=& \left(\dfrac{1}{2},\dfrac{1}{3}, \dfrac{1}{k}\right)
\end{eqnarray*}
and hence 
\begin{eqnarray*}
& & (\rot(\phi_{\mathcal{O}_{2,3,k}}(\alpha)), \rot(\phi_{\mathcal{O}_{2,3,k}}(\beta)), \rot(\phi_{\mathcal{O}_{2,3,k}}(\alpha\beta))) \\
&=& \left(\dfrac{1}{2},\dfrac{1}{3}, \dfrac{k-1}{k}\right)
\end{eqnarray*}
Let $q$ be the homomorphism from $\pi_1^{orb}(\mathcal{O}_{2,3})$ onto $\pi_1^{orb}(\mathcal{O}_{2,3,k})$ such that $q(\alpha)=\alpha$ and $q(\beta)=\beta$ and let $\iota$ be the automorphism of $\pi_1^{orb}(\mathcal{O}_{2,3})$ such that $\iota(\alpha)=\alpha$ and $\iota(\beta)=\beta^{-1}$. 
We define a homomorphism $\hat{\phi}_{\mathcal{O}_{2,3,k}}\in\Rep(\pi_1^{orb}(\mathcal{O}_{2,3}))$ by
\begin{eqnarray*}
\hat{\phi}_{\mathcal{O}_{2,3,k}}=\left\{
\begin{array}{ll}
\phi_{\mathcal{O}_{2,3,k}}\circ q & (k\equiv1 \mod 6) \\
\phi_{\mathcal{O}_{2,3,k}}\circ q\circ\iota & (k\equiv-1 \mod 6),
\end{array}
\right.
\end{eqnarray*}
Since both $\phi_{\mathcal{O}_{2,3,k}}$ and $\phi_{\mathcal{O}_{2,3}}$ are minimal, it follows that both $\hat{\phi}_{\mathcal{O}_{2,3,k}}$ are $\phi_{\mathcal{O}_{2,3}}^{(k)}$ are also minimal. 
It follows that 
\begin{eqnarray*}
& & (\rot(\hat{\phi}_{\mathcal{O}_{2,3,k}}(\alpha)), \rot(\hat{\phi}_{\mathcal{O}_{2,3,k}}(\beta)), \rot(\hat{\phi}_{\mathcal{O}_{2,3,k}}(\alpha\beta))) \\
&=& (\rot(\phi_{\mathcal{O}_{2,3}}^{(k)}(\alpha)), \rot(\phi_{\mathcal{O}_{2,3}}^{(k)}(\beta)), \rot(\phi_{\mathcal{O}_{2,3}}^{(k)}(\alpha\beta))). 
\end{eqnarray*}
Note that if $k\equiv-1 \mod 6$, then we have 
\begin{eqnarray*}
& & \rot(\hat{\phi}_{\mathcal{O}_{2,3,k}}(\alpha\beta)) \\
&=& \rot(\phi_{\mathcal{O}_{2,3,k}}(\alpha\beta^{-1})) \\
&=& \rot(\phi_{\mathcal{O}_{2,3,k}}(\beta)(\phi_{\mathcal{O}_{2,3,k}}(\alpha\beta))^{-1}(\phi_{\mathcal{O}_{2,3,k}}(\beta))^{-1}) \\
&=& -\rot(\phi_{\mathcal{O}_{2,3,k}}(\alpha\beta)).
\end{eqnarray*}
On the other hand $\hat{\phi}_{\mathcal{O}_{2,3,k}}$ and $\phi_{\mathcal{O}_{2,3}}^{(k)}$ do not belong to the same semi-conjugacy class. 
Indeed if they belonged the same conjugacy class, then they would be topologically conjugate by minimality. 
However this contradicts the fact that 
\begin{eqnarray*}
\hat{\phi}_{\mathcal{O}_{2,3,k}}((\alpha\beta)^k)=\id\neq\phi_{\mathcal{O}_{2,3}}^{(k)}((\alpha\beta)^k). 
\end{eqnarray*}

\item We can prove Theorem \ref{main} (1) by generalizing the notion of the bounded Euler number defined in \cite{BIW14} to actions of 2-orbifold groups. 
It will be indicated in a forthcoming paper together with generalizations of Theorem \ref{main} to actions of other 2-orbifold groups. 

\item Theorem \ref{main} can be considered as a weak analogue of the following classical theorem about linear character \cite{FK1897}, which we write in a specified form. 
Let $F\langle\alpha,\beta\rangle$ be a free group of rank two with a basis $\alpha, \beta$. 

\begin{Thm}
\label{FK}
Let $\phi,\psi \colon F\langle\alpha,\beta\rangle\to\SL(2,\R)$ be homomorphisms. If we have 
\begin{eqnarray*}
& &(\Tr(\phi(\alpha)),\Tr(\phi(\beta)),\Tr(\phi(\alpha\beta))) \\
&=&(\Tr(\psi(\alpha)),\Tr(\psi(\beta)),\Tr(\psi(\alpha\beta))) \\
&=&(x,y,z)
\end{eqnarray*}
with $x^2+y^2+z^2-xyz\neq 4$, then $\phi$ and $\psi$ are conjugate by an element of $\PSL(2,\R)$. 
\end{Thm}

\item When the author mentioned Theorem \ref{main} in his talk given in the conference ``Geometry and Foliations 2013'', E. Ghys informed us the following theorem about linear character. 

\begin{Thm}\cite[Example 8.2]{Hor72}
\label{Hor}
Let $F_m$ be a free group of rank $m\ge 2$. 
For every positive integer $n$, there exist mutually non-conjugate elements $w_1,\ldots,w_n$ of $F_m$ such that for every homomorphism $\phi\colon F_m\to\SL(2,\R)$, we have 
\begin{eqnarray*}
\Tr(\phi(w_1))=\cdots=\Tr(\phi(w_n)). 
\end{eqnarray*}
\end{Thm}

After that, he asked the following question. 

\begin{Quest}
\label{Ghys}
Does the following analogue of Theorem \ref{Hor} hold for $\Homeo_+(\mathbb{S}^1)$? 
Namely, for every positive integer $m\ge2$ and every positive integer $n$, does there exist mutually non-conjugate elements $w_1,\ldots,w_n$ of $F_m$ such that for every homomorphism $\phi\in\Rep(F_m)$, we have 
\begin{eqnarray*}
\rot(\phi(w_1))=\cdots=\rot(\phi(w_n))\ ?
\end{eqnarray*}
\end{Quest}

Note that D. Calegari asked this question for the case where $m=2, n=2$ and $w_2$ is fixed as the identity element \cite{Cal12}. 
\end{enumerate}
\end{Rem}

\section{Proof of Theorem \ref{main}}
For $r_1, r_2, r_3\in\R\slash\Z$, we put 
\begin{eqnarray*}
& &\Rep(r_1, r_2, r_3) \\
&=& \{\phi\in\Rep(\pi_1^{orb}(\mathcal{O}_{2,3})) ~|~ (\rot(\phi(\alpha)), \rot(\phi(\beta)), \rot(\phi(\alpha\beta))=(r_1, r_2, r_3)\}. 
\end{eqnarray*}

\subsection{Proof of (1)}
Let $\phi\in\Rep\left(\dfrac{1}{2},\dfrac{1}{3}, 0\right)$. 
The following sufficient condition for belonging to the same semi-conjugacy class given in \cite{Man14} is a corollary of a criterion in \cite{Mat86}. 

\begin{Prop}
\cite[Corollary 7.5]{Man14}
\label{rotconn}
Let $\Gamma$ be a group and $U\subset\Rep(\Gamma)$ be connected. 
Suppose that $\rot(\phi_1(\gamma))=\rot(\phi_2(\gamma))$ for every $\phi_1,\phi_2\in U$ and every $\gamma\in\Gamma$, then $U$ is contained in a single semi-conjugacy class.
\end{Prop}

In view of Proposition \ref{rotconn}, it suffices to show the following. 

\begin{Lem}
\label{rot}
$\rot(\phi(\gamma))=\rot(\phi_{\mathcal{O}_{2,3}}(\gamma))$ for every $\gamma\in\pi_1^{orb}(\mathcal{O}_{2,3})$. 
\end{Lem}

\begin{Lem}
\label{pathconn}
The space $\Rep\left(\dfrac{1}{2},\dfrac{1}{3}, 0\right)$ is path-connected.
\end{Lem}

\begin{proof}[Proof of \textup{Lemma \ref{rot}}]

We denote by $\tilde{a}$ (resp.\ $\tilde{b}$) the lift of $\phi(\alpha)$ (resp.\ $\phi(\beta)$) with $\widetilde{\rot}(\tilde{a})=\dfrac{1}{2}$ $\left({\rm{resp.}}\ \widetilde{\rot}(\tilde{b})=\dfrac{1}{3}\right)$. 
Since $0<\widetilde{\rot}(\tilde{a})<1$, we have
\begin{eqnarray*}
\tilde{x}<\tilde{a}(\tilde{x})<\tilde{x}+1
\end{eqnarray*}
for every $\tilde{x}\in\mathbb{R}$. 
Hence we have
\begin{eqnarray*}
\tilde{b}(\tilde{x})<(\tilde{a}\tilde{b})(\tilde{x})<\tilde{b}(\tilde{x})+1
\end{eqnarray*}
for every $\tilde{x}\in\mathbb{R}$. 
This implies that
\begin{eqnarray*}
\dfrac{1}{3}=\widetilde{\rot}(\tilde{b})\le\widetilde{\rot}(\tilde{a}\tilde{b})\le\widetilde{\rot}(\tilde{b})+1=\dfrac{4}{3}. 
\end{eqnarray*}
Since $\rot(\phi(\alpha\beta))=0$, we have $\widetilde{\rot}(\tilde{a}\tilde{b})=1$. 
Then there exists a point $\tilde{x_0}\in\R$ such that $(\tilde{a}\tilde{b})(\tilde{x}_0)=\tilde{x}_0+1$. 
Since both $\tilde{a}^2$ and $\tilde{b}^3$ are the translation by one, we have
\begin{eqnarray*}
\tilde{x}_0<\tilde{a}(\tilde{x}_0)=\tilde{b}(\tilde{x}_0)<\tilde{b}^2(\tilde{x}_0)<\tilde{x}_0+1. 
\end{eqnarray*}
We put 
\begin{eqnarray*}
I &=& p([\tilde{x}_0,\tilde{b}(\tilde{x}_0)]\ {\textrm{and}} \\
J &=& p([\tilde{b}(\tilde{x}_0),\tilde{x}_0+1]). 
\end{eqnarray*}
Then we have 
\begin{eqnarray*}
\phi(\alpha)(J) &=& I \ {\textrm{and}} \\
\phi(\beta^{\pm 1})(I) & \subset & J. 
\end{eqnarray*}

We claim that if $\gamma\in\Gamma$ is not conjugate to a power of $\alpha, \beta$, then there exists a closed interval $K\subset \mathbb{S}^1$ such that $\phi(\gamma)(K)\subset K$. 
Indeed by taking conjugates if necessary, we may assume that $\gamma=\alpha\beta^{e_1}\cdots \alpha\beta^{e_n}$, where $e_i\in{\pm 1}$ for $i\in\{1,\ldots,n\}$. 
Then we have $\phi(\gamma)(I)\subset I$. 

This implies that if $\gamma$ is not conjugate to a power of $\alpha, \beta$, then $\rot(\phi(\gamma))=0$. 
This finishes the proof of the lemma.
\end{proof}

\begin{proof}[Proof of \textup{Lemma \ref{pathconn}}]
Let $\phi_0,\phi_1\in\Rep\left(\dfrac{1}{2},\dfrac{1}{3}, 0\right)$. 
We show that there exists a path in $\Rep\left(\dfrac{1}{2},\dfrac{1}{3}, 0\right)$ from $\phi_0$ to $\phi_1$. 
For $t\in\{0,1\}$, we denote by $\tilde{a}_t$ (resp.\ $\tilde{b}_t$) the lift of $\phi_t(\alpha)$ (resp.\ $\phi_t(\beta)$) with $\widetilde{\rot}(\tilde{a}_t)=\dfrac{1}{2}$\ $\left({\rm{resp.}}\ \widetilde{\rot}(\tilde{b}_t)=\dfrac{1}{3}\right)$. 
By taking conjugates, we may assume that both $\phi_0(b)$ and $\phi_1(b)$ are the rotation by $\dfrac{1}{3}$, and that $(\tilde{a}_t\tilde{b}_t)(0)=1$ for $t\in\{0,1\}$. 
We take a path $\{\tilde{a}_t\}_{t\in[0,1]}$ in $\widetilde{\Homeo}_+(\mathbb{S}^1)$ from $\tilde{a}_0$ to $\tilde{a}_1$ such that $(\tilde{a}_t)\left(\dfrac{1}{3}\right)=1$ and $(\tilde{a}_t)^2$ is the translation by one. 
We denote by $a_t\in\Homeo_+(\mathbb{S}^1)$ the projection of $\tilde{a}_t$. 
Then the path $\{\phi_t\}_{t\in[0,1]}$ in $\Rep\left(\dfrac{1}{2},\dfrac{1}{3}, 0\right)$ defined by the condition that $\phi_t(\alpha)=a_t$ and $\phi_t(\beta)$ is the rotation by $\dfrac{1}{3}$ is a desired one. 
\end{proof}

\subsection{Proof of (2)}
Let $\phi\in\Rep\left(\dfrac{1}{2},\dfrac{2}{3}, \dfrac{1}{5}\right)$. 
Then $\phi$ has no finite orbits. 
In fact if there were finite orbits, then the map $\rot\circ\phi \colon \Z_2\ast\Z_3 \to \R\slash\Z$ must be a homomorphism, which is impossible since $\rot(\phi(\alpha))=\dfrac{1}{2}$, $\rot(\phi(\beta))=\dfrac{2}{3}$ and $\rot(\phi(\alpha\beta))=\dfrac{1}{5}$. 
Therefore the action $\phi$ admits a unique minimal set, either a Cantor set or the whole circle. 
Passing to a semi-conjugate action, we may assume the latter, that is, the action is minimal.

By Theorem \ref{main} (1), it suffices to show that $\phi$ is the 5-fold lift of some action, namely, there exists a homeomorphism $\theta\in\Homeo_+(\mathbb{S}^1)$ which is $\phi(\pi_1^{orb}(\mathcal{O}_{2,3}))$-equivariant and periodic of period 5. 

We denote by $\tilde{a}$ (resp. $\tilde{b}$) the lift of $\phi(\alpha)$ (resp. $\phi(\beta)$) with $\widetilde{\rot}(\tilde{a})=\dfrac{1}{2}$\ $\left({\rm{resp.}}\ \widetilde{\rot}(\tilde{b})=\dfrac{2}{3}\right)$. 
Since $0<\widetilde{\rot}(\tilde{a})<1$, we have
\begin{eqnarray*}
\tilde{x}<\tilde{a}(\tilde{x})<\tilde{x}+1
\end{eqnarray*}
for every $\tilde{x}\in\mathbb{R}$. 
Hence we have
\begin{eqnarray*}
\tilde{b}(\tilde{x})<(\tilde{a}\tilde{b})(\tilde{x})<\tilde{b}(\tilde{x})+1
\end{eqnarray*}
for every $\tilde{x}\in\mathbb{R}$. 
This implies that
\begin{eqnarray*}
\dfrac{2}{3}=\widetilde{\rot}(\tilde{b})\le\widetilde{\rot}(\tilde{a}\tilde{b})\le\widetilde{\rot}(\tilde{b})+1=\dfrac{5}{3}. 
\end{eqnarray*}
Since $\rot(\phi(\alpha\beta))=\dfrac{1}{5}$, we have $\widetilde{\rot}(\tilde{a}\tilde{b})=\dfrac{6}{5}$. 
We denote by $\widetilde{ab}$ the lift of $\phi(\alpha\beta)$ with $\widetilde{\rot}(\widetilde{ab})=\dfrac{1}{5}$. 
Then there exists a point $\tilde{x_0}\in\R$ such that $(\widetilde{ab})^5(\tilde{x}_0)=\tilde{x}_0+1$. 
Note that $\tilde{a}\tilde{b}(\tilde{x})=\widetilde{ab}(\tilde{x})+1$ for every $\tilde{x}\in\R$. 
\begin{Lem}
\label{ineq}
We have the following.
\begin{eqnarray*}
&(1)& 
\tilde{a}(\tilde{x})
<\tilde{b}(\tilde{x})
\ for\ every\ \tilde{x}\in\R. \\
&(2)& 
(\widetilde{ab})^2\tilde{a}(\tilde{x})
<\tilde{x}+1\ for\ every\ \tilde{x}\in\R. \\
&(3)& 
(\widetilde{ab})^l(\tilde{x}_0)
<\tilde{b}(\widetilde{ab})^{l+2}(\tilde{x}_0)-1
<\tilde{b}^2(\widetilde{ab})^{l+4}(\tilde{x}_0)-2
<(\widetilde{ab})^{l+1}(\tilde{x}_0) \\
& & 
for\ every\ l\in\Z. 
\end{eqnarray*}
\end{Lem}
\begin{proof}
(1) Since $\widetilde{\rot}(\tilde{a}\tilde{b})=\dfrac{6}{5}>1$, we have 
\begin{eqnarray*}
\tilde{a}^2(\tilde{x})=\tilde{x}+1<\tilde{a}\tilde{b}(\tilde{x})
\end{eqnarray*}
for every $\tilde{x}\in\R$. 
This implies the desired inequality. \\
(2) It follows from (1) that for every $\tilde{x}\in\R$ we have 
\begin{eqnarray*}
(\widetilde{ab})^2\tilde{a}(\tilde{x})
=(\tilde{a}\tilde{b})^2\tilde{a}(\tilde{x})-2
<\tilde{a}\tilde{b}^3\tilde{a}(\tilde{x})-2
=\tilde{a}^2(\tilde{x})
=\tilde{x}+1. \\
\end{eqnarray*}
(3) By substituting $\tilde{b}(\widetilde{ab})^{l+2}(\tilde{x}_0)$ for $\tilde{x}$ in inequality (2), it follows that 
\begin{eqnarray*}
(\widetilde{ab})^2\tilde{a}\tilde{b}(\widetilde{ab})^{l+2}(\tilde{x}_0)
<\tilde{b}(\widetilde{ab})^{l+2}(\tilde{x}_0)+1.
\end{eqnarray*}
Since we have 
\begin{eqnarray*}
(\widetilde{ab})^2\tilde{a}\tilde{b}(\widetilde{ab})^2((\widetilde{ab})^l(\tilde{x}_0))
=(\widetilde{ab})^5((\widetilde{ab})^l(\tilde{x}_0))+1
=(\widetilde{ab})^l(\tilde{x}_0)+2,
\end{eqnarray*}
we obtain the first inequality. 
Since $l\in\Z$ is an arbitrary integer, it follows that
\begin{eqnarray*}
(\widetilde{ab})^{l+2}(\tilde{x}_0)
<\tilde{b}(\widetilde{ab})^{l+4}(\tilde{x}_0)-1.
\end{eqnarray*}
This implies the second inequality. 
Similarly we have
\begin{eqnarray*}
(\widetilde{ab})^{l+4}(\tilde{x}_0)
<\tilde{b}(\widetilde{ab})^{l+6}(\tilde{x}_0)-1
=\tilde{b}(\widetilde{ab})^{l+1}(\tilde{x}_0).
\end{eqnarray*}
This implies the third inequality.
\end{proof}

The following lemma follows from Lemma \ref{ineq} (3) and the equality $\tilde{a}(\widetilde{ab})^l(\tilde{x}_0)=\tilde{b}(\widetilde{ab})^{l+4}(\tilde{x}_0)-1$. 
\begin{Lem}
\label{Markov}
For every integer $l\in\Z$, we put 
\begin{eqnarray*}
\tilde{I}_l &=& ((\widetilde{ab})^l(\tilde{x}_0), (\tilde{b}(\widetilde{ab})^{l+2})(\tilde{x}_0)-1]\ {\textrm{and}} \\
\tilde{J}_l &=& ((\tilde{b}(\widetilde{ab})^{l+2})(\tilde{x}_0)-1,(\widetilde{ab})^{l+1}(\tilde{x}_0)]. 
\end{eqnarray*}
Then we have the following. 
\begin{enumerate}
\item $
\begin{array}{lllll}
\tilde{b}^{-1}((\widetilde{ab})^l(\tilde{x}_0)) &\in& \Int(\tilde{J}_{l-4}) \ {\textrm{and}} \\
(\tilde{b}\tilde{a})((\widetilde{ab})^l(\tilde{x}_0)) &\in& \Int(\tilde{J}_{l+5}). 
\end{array}
$
\item $
\begin{array}{lllll}
\tilde{a}(\tilde{J}_l) &=& \tilde{I}_{l+3}, \\
\tilde{b}(\tilde{I}_l) &\subset& \tilde{J}_{l+3} \ {\textrm{and}}\ \tilde{b}^{-1}(\tilde{I}_l)\ &\subset& \tilde{J}_{l-4}. 
\end{array}
$
\end{enumerate}
\end{Lem}

We denote by $\widetilde{\phi(\pi_1^{orb}(\mathcal{O}_{2,3}))}$ the subgroup of $\widetilde{\Homeo}_+(\mathbb{S}^1)$ consisting of lifts of elements of $\phi(\pi_1^{orb}(\mathcal{O}_{2,3}))$ to $\R$. 
We define a map $\tilde{\theta}$ of $\widetilde{\phi(\pi_1^{orb}(\mathcal{O}_{2,3}))}(\tilde{x}_0)$ onto itself by 
\begin{eqnarray*}
\tilde{\theta}(\widetilde{\phi(\gamma)}(\tilde{x}_0))=\widetilde{\phi(\gamma)}(\widetilde{ab}(\tilde{x}_0)),
\end{eqnarray*}
where $\gamma\in\pi_1^{orb}(\mathcal{O}_{2,3})$ and $\widetilde{\phi(\gamma)}$ is a lift of $\phi(\gamma)$ to $\R$. 

\begin{Lem}
The map $\tilde{\theta}$ is well-defined and strictly increasing.
\end{Lem}
\begin{proof}
First we prove that $\tilde{\theta}$ is well-defined. 
It suffices to show that for $\widetilde{\phi(\gamma)}\in\widetilde{\phi(\pi_1^{orb}(\mathcal{O}_{2,3}))}$ with $\widetilde{\phi(\gamma)}(\tilde{x}_0)=\tilde{x}_0$, we have $\widetilde{\phi(\gamma)}(\widetilde{ab}(\tilde{x}_0))=\widetilde{ab}(\tilde{x}_0)$. 

If $\gamma=\beta^{e_0}\alpha\beta^{e_1}\cdots \alpha\beta^{e_n}$, where $e_0\in\{0,\pm 1\}$ and $e_i\in\{\pm 1\}$ for $i\in\{1,\ldots,n\}$, then we have $e_i\neq -1$ for $i\in\{0,1,\ldots,n\}$. 
Indeed if $(e_i,e_{i+1},\ldots,e_n)=(-1,1,\ldots,1)$ for some $i\in\{0,1,\ldots,n\}$, then it would follow from Lemma \ref{Markov} (1) that 
\begin{eqnarray*}
\tilde{b}^{e_i}\cdots \tilde{a}\tilde{b}^{e_n}(\tilde{x}_0)
&=& \tilde{b}^{-1}(\tilde{a}\tilde{b})^{n-i}(\tilde{x}_0)) \\
&=& \tilde{b}^{-1}((\widetilde{ab})^{n-i}(\tilde{x}_0))+(n-i)
\in\Int(\tilde{J}_{6(n-i)-4})
\end{eqnarray*}
and hence
\begin{eqnarray*}
\widetilde{\phi(\gamma)}(\tilde{x}_0)\in\Int(\tilde{I}_{l})\cup\Int(\tilde{J}_{l})
\end{eqnarray*}
for some $l\in\Z$ by Lemma \ref{Markov} (2), which contradicts the assumption. 

\noindent
Therefore we have $\gamma=\beta^{e_0}(\alpha\beta)^n$, where $e_0\in\{0,1\}$ and it follows from Lemma \ref{ineq} (3) we have $e_0\neq 1$. 
Hence  there exists an integer $m\in\Z$ such that 
\begin{eqnarray*}
\widetilde{\phi(\gamma)}(\tilde{x})=(\widetilde{ab})^n(\tilde{x})+m
\end{eqnarray*}
for every $\tilde{x}\in\R$. 
We have $n=-5m$ by the assumption and hence 
\begin{eqnarray*}
\widetilde{\phi(\gamma)}(\widetilde{ab}(\tilde{x}_0))
=(\widetilde{ab})^{-5m+1}(\tilde{x}_0)+m
=\widetilde{ab}(\tilde{x}_0).
\end{eqnarray*}

If $\gamma=\beta^{e_0}\alpha\beta^{e_1}\cdots \alpha\beta^{e_n}\alpha$, where $e_0\in\{0,\pm 1\}$ and $e_i\in\{\pm 1\}$ for $i\in\{1,\ldots,n\}$, then we have $e_i\neq 1$ for $i\in\{0,1,\ldots,n\}$. 
Indeed if $(e_i,e_{i+1},\ldots,e_n)=(1,-1,\ldots,-1)$ for some $i\in\{0,1,\ldots,n\}$, then it would follow from Lemma \ref{Markov} (1) that 
\begin{eqnarray*}
\tilde{b}^{e_i}\cdots \tilde{a}\tilde{b}^{e_n}(\tilde{x}_0)
&=&(\tilde{b}\tilde{a})(\tilde{b}^{-1}\tilde{a})^{n-i}(\tilde{x}_0)) \\
&=&(\tilde{b}\tilde{a})((\widetilde{ab})^{-(n-i)}(\tilde{x}_0))
\in\Int(\tilde{J}_{-(n-i)+5})
\end{eqnarray*}
and hence
\begin{eqnarray*}
\widetilde{\phi(\gamma)}(\tilde{x}_0)\in\Int(\tilde{I}_{l})\cup\Int(\tilde{J}_{l})
\end{eqnarray*}
for some $l\in\Z$ by Lemma \ref{Markov} (2), which contradicts the assumption. 

\noindent
Therefore we have $\gamma=\beta^{e_0}\alpha(\beta\alpha)^{n-1}$, where $e_0\in\{0,-1\}$ and it follows from Lemma \ref{ineq} (3) that we have $e_0\neq 0$. 
Hence there exists an integer $m\in\Z$ such that 
\begin{eqnarray*}
\widetilde{\phi(\gamma)}(\tilde{x})=(\widetilde{ab})^{-(n+1)}(\tilde{x})+m
\end{eqnarray*}
for every $\tilde{x}\in\R$. 
We have $n=5m-1$ by the assumption and hence 
\begin{eqnarray*}
\widetilde{\phi(\gamma)}(\widetilde{ab}(\tilde{x}_0))=(\widetilde{ab})^{-5m+1}(\tilde{x}_0)+m=\widetilde{ab}(\tilde{x}_0).
\end{eqnarray*}

Next we prove that $\tilde{\theta}$ is strictly increasing. 
It suffices to show that for $\widetilde{\phi(\gamma)}\in\widetilde{\phi(\pi_1^{orb}(\mathcal{O}_{2,3}))}$ with $\tilde{x}_0 < \widetilde{\phi(\gamma)}(\tilde{x}_0)$, we have $\tilde{\theta}(\tilde{x}_0) < \tilde{\theta}(\widetilde{\phi(\gamma)}(\tilde{x}_0))$. 

If $\gamma=\beta^{e_0}\alpha\beta^{e_1}\cdots \alpha\beta^{e_n}$, where $e_0\in\{0,\pm 1\}$ and $e_i\in\{\pm 1\}$ for $i\in\{1,\ldots,n\}$, then it follows from Lemma \ref{Markov} (2) that 
\begin{eqnarray*}
\widetilde{\phi(\gamma)}(\tilde{I}_0)\subset \tilde{I}_{l}\cup\tilde{J}_{l}
\end{eqnarray*}
for some non-negative integer $l\in\Z$. 
This implies that 
\begin{eqnarray*}
\widetilde{\phi(\gamma)}(\tilde{I}_1)\subset \tilde{I}_{l+1}\cup\tilde{J}_{l+1}
\end{eqnarray*}
and hence $\tilde{\theta}(\tilde{x}_0) < \tilde{\theta}(\widetilde{\phi(\gamma)}(\tilde{x}_0))$. 

If $\gamma=\beta^{e_0}\alpha\beta^{e_1}\cdots \alpha\beta^{e_n}\alpha$, where $e_0\in\{0,\pm 1\}$ and $e_i\in\{\pm 1\}$ for $i\in\{1,\ldots,n\}$, then it follows from Lemma \ref{Markov} (2) that 
\begin{eqnarray*}
\widetilde{\phi(\gamma)}(\tilde{J}_{-1})\subset \tilde{I}_{l}\cup\tilde{J}_{l}
\end{eqnarray*}
for some non-negative integer $l\in\Z$. 
This implies that 
\begin{eqnarray*}
\widetilde{\phi(\gamma)}(\tilde{J}_0)\subset \tilde{I}_{l+1}\cup\tilde{J}_{l+1}
\end{eqnarray*}
and hence $\tilde{\theta}(\tilde{x}_0) < \tilde{\theta}(\widetilde{\phi(\gamma)}(\tilde{x}_0))$. 
\end{proof}

The map $\tilde{\theta}$ is $\widetilde{\phi(\pi_1^{orb}(\mathcal{O}_{2,3}))}$-equivariant and we have $\tilde{\theta}^5(\widetilde{\phi(\gamma)}(\tilde{x}_0))=\widetilde{\phi(\gamma)}(\tilde{x}_0)+1$ for every element $\widetilde{\phi(\gamma)}$ of $\widetilde{\phi(\pi_1^{orb}(\mathcal{O}_{2,3}))}$. 
Since $\phi$ is minimal, $\widetilde{\phi(\pi_1^{orb}(\mathcal{O}_{2,3}))}(\tilde{x}_0)$ is dense in $\R$ and hence $\tilde{\theta}$ can be extended to an element of $\widetilde{\Homeo}_+(\mathbb{S}^1)$, which we also denote by $\tilde{\theta}$. 
The homeomorphism $\tilde{\theta}$ is $\widetilde{\phi(\pi_1^{orb}(\mathcal{O}_{2,3}))}$-equivariant and we have $\tilde{\theta}^5(\tilde{x})=\tilde{x}+1$ for every $\tilde{x}\in\R$. 
This gives the desired homeomorphism $\theta\in\Homeo_+(\mathbb{S}^1)$. \\

\textbf{Acknowledgements}\ The author would like to thank the referee for his/her careful reading of the manuscript and helpful comments.

\end{document}